\documentclass{article}
\usepackage[utf8]{inputenc}
\usepackage{commands}
\usepackage[margin=1in]{geometry}
\usepackage[title]{appendix}
\usepackage{biblatex}
\usepackage{asymptote}

\allowdisplaybreaks[1]

\title{Families with no perfect matchings}
\author{Mihir Singhal\thanks{Massachusetts Institute of Technology, Cambridge, MA 02139. Email: \href{mailto:mihirs@mit.edu}{\nolinkurl{mihirs@mit.edu}}.}}
\date{} %TODO

\begin{document}

\maketitle

\begin{abstract}
We consider families of $k$-subsets of $\{1, \dots, n\}$, where $n$ is a multiple of $k$, which have no perfect matching. An equivalent condition for a family $\mcF$ to have no perfect matching is for there to be a \textit{blocking set}, which is a set of $b$ elements of $\{1, \dots, n\}$ that cannot be covered by $b$ disjoint sets in $\mcF$. We are specifically interested in the largest possible size of a family $\mcF$ with no perfect matching and no blocking set of size less than $b$. Frankl resolved the case of families with no singleton blocking set (in other words, the $b=2$ case) for sufficiently large $n$ and conjectured an optimal construction for general $b$. Though Frankl's construction fails to be optimal for $k = 2, 3$, we show that the construction is optimal whenever $k \ge 100$ and $n$ is sufficiently large. 
\end{abstract}

\section{Introduction}
Let $n = ks$ where $k, s \ge 2$. Consider a family $\mcF$ of $k$-element subsets of $[n] = \{1, \dots, n\}$. A \textit{matching} in $\mcF$ is a collection of sets in $\mcF$ which are all pairwise disjoint. A \textit{perfect matching} is one which covers all elements of $[n]$ (equivalently, a matching of $k$-sets is a perfect matching if it has size $s$). We are interested in families $\mcF$ with no perfect matching. Specifically, we explore the question about how large such a family can be. This question is related to the Erd\H{o}s matching conjecture \cite{erdos}, which regards the largest possible size of a family with no matching of size $r$, where $n$ is not necessarily equal to $rk$.

Kleitman \cite{kleitman} showed that largest possible size of $\mcF$ given that it has no perfect matching is $\binom{n-1}{k}$, which is achieved by any family $\mcF$ which consists of all $k$-sets that do not contain a single vertex. Frankl \cite{frankl-shifting} later showed that these are the only families which achieve this maximum. Such families can be considered \textit{trivial} in the sense that there is a vertex in none of the sets of $\mcF$, which immediately implies that $\mcF$ does not have a perfect matching. Frankl considered the question of finding the largest possible size of a nontrivial $\mcF$ which has no perfect matching, and proved the following:

\begin{thm}[\cite{frankl-main}] \label{thm:frankl}
Suppose $s$ is sufficiently large (relative to $k$). Then, if $\mcF$ is a nontrivial family of $k$-subsets of $[n]$, then the following family of $k$-sets attains the maximum possible size of $\mcF$:
\[\mc E(k, n) = \big\{\{1\} \cup [n-k+2, n]\big\} \cup \big\{S: S \cap [n-k+2, n] \neq \emptyset,\: 1 \notin S,\: 2 \in S\big\} \cup \big\{T: T \subset [3, n]\big\}.\]
\end{thm}

The family $\mc E(k, n)$ has no matching for the simple reason that there are no two disjoint sets in $\mc E(k, n)$ which cover both $1$ and $2$ (in particular, there is also no single set which covers both). This is an example of the notion of a \textit{blocking set}. 

In general, a \textit{blocking set} of $\mcF$ is defined to be any set $B \subset [n]$ of size at most $s$ which cannot be covered by a matching of size $|B|$. Note that a blocking set also cannot be covered by a larger matching since one can delete sets from any such matching which do not intersect $B$. (One can also define a blocking set to be one that is not coverable by any matching, regardless of its size; we will later see in \cref{lem:coveringmatching} that it does not matter for our purposes which definition we choose.) It is clear that if $\mcF$ has a blocking set, then it cannot have a perfect matching; a small blocking set might also be considered to be a ``trivial reason" that $\mcF$ does not have a matching. Indeed, as shown in \cref{thm:frankl}, the smallest nontrivial $\mcF$ with no perfect matching has a blocking set of size $2$. It is thus natural to ask what the largest possible size of $\mcF$ is if it has no small blocking set. Indeed, Frankl posed the following question:
\begin{question}[\cite{frankl-main}]
Suppose $\mcF$ has no perfect matching and no blocking set of size less than $b$. What is the largest possible size of $\mcF$?
\end{question}
The case $b=1$ is equivalent to having no restriction on blocking sets of $\mcF$, so Kleitman's bound of $\binom{n-1}{k}$ is the maximum. The $b=2$ case restricts to nontrivial $\mcF$, which is just \cref{thm:frankl} (for large enough $n$). In fact, Frankl described a construction of a family for general $b$ which he conjectured to be optimal. The construction is as follows.
\begin{example}[\cite{frankl-main}] \label{ex:opt}
Suppose $n \ge bk + b$. Let $E_1, \dots, E_{b-1}$ be pairwise disjoint $(k-1)$-sets, none of which intersect $[b]$. Define the families $\mc{E}_i$, for $1 \le i \le b$, as follows:
\[\mc{E}_i = \big\{\{i\} \cup E_i\big\} \cup \big\{\{i\} \cup S: S \subset [b+1,n],\: |S| = k-1,\: S \cap (E_1 \cup \dots \cup E_{i-1}) \neq \emptyset\big\}.\]
(If $i = b$, then ignore the $\{\{i\} \cup E_i\}$.) Then we define the family
\[\mc E(k, n, b) = \big\{T: T \subseteq [b+1, n],\: |T| = k\big\} \cup \left(\bigcup_{i=1}^b \mc{E}_i \right).\]
If there is any matching in $\mc E(k, n, b)$ covering $[b]$, then by induction on $i$ it will need to contain the set $\{i\} \cup E_i$ for each $i$. But there is no such set for $i=b$, a contradiction, so there can be no matching covering $[b]$. Therefore $[b]$ is a blocking set of $\mc E(k, n, b)$, and it is also easy to check that there is no smaller blocking set. It is also simple to check that when $k > 2$, this family $\mc E(k, n, b)$ is maximal in the sense that if any sets are added to it, then there will be a perfect matching.
\end{example}

When $b=2$, the set $\mc E(k, n, 2)$ is the same (up to permutation of $[n]$) as $\mc E(k, n)$. Indeed, Frankl conjectured that the above construction is optimal when $n$ is large:
\begin{conj}[\cite{frankl-main}] \label{conj:main}
Suppose that $n = ks$ and $k, s \ge 2$ and $b$ is a positive integer. There exists $s_0(k, b)$ such that whenever $s \ge s_0(k, b)$, the following holds: If $\mc F$ is a family of $k$-subsets of $[n]$ with no perfect matching and no blocking set of size less than $b$, then
\[|\mcF| \le |\mc E(k, n, b)|.\]
\end{conj}

It turns out that \cref{conj:main} actually fails when $b > 2$ and $k$ is 2 or 3. When $k=2$, $\mcF$ is just an ordinary graph, and it turns out that one can just add edges to $\mc E(2, n, b)$. In particular, noting that $E_i$ are all singleton sets, we can add all the edges between each $i \in [b]$ and each $E_j$ (only some of these edges exist in $\mc E(2, n, b)$). Then, elements of $[b]$ are only adjacent to elements of a $(b-1)$-element set, so there is still no matching covering $[b]$.

When $k=3$, the situation is a little more subtle, but there is still a construction that has slightly more sets than $\mc E(3, n, b)$.

\begin{example} \label{ex:k3}
Suppose $k=3$ and $n \ge bk+b$. Then define the family of $k$-sets
\[\mc E'(3, n, b) = \big\{S: |S \cap [b]| = 1,\: |S \cap [b+1,2b-1]| \ge 1\big\} \cup \big\{T: T \subset [b+1, n] \big\}.\]
The set $[b]$ is a blocking set of $\mc E'(3, n, b)$ because any matching which covers $[b]$ would need to contain one set which covers each element of $[b]$, and each such set would need to contain at least one element of $[b+1, 2b-1]$. Thus these $b$ sets cannot be pairwise disjoint, so no such matching can exist. It is also easy to check that there is no smaller blocking set. Then, the size of $\mc E'(3, n, b)$ exceeds that of $\mc E(3, n, b)$ by $(b^3-7b+6)/6$, which is positive whenever $b>2$. We defer the calculation to verify this to \cref{sec:sizeE}.
\end{example}

Our main theorem is that \cref{conj:main} holds whenever $k \ge 100$, and also when $k \ge \max \{b+1, 6\}$. Specifically, there exists a function $s_0(k, b)$ such that the following theorem holds. 

\begin{thm} \label{thm:main}
Let $k, b$ be positive integers such that either $k \ge 100$ or $k \ge \max\{b+1, 6\}$. Also let $n = ks$, where $s \ge s_0(k, b)$. Suppose $\mcF$ is a family of subsets of $[n]$ of size $k$. Then, if $\mcF$ has no perfect matching and no blocking set of size less than $b$, then
\[|\mcF| \le |\mc E(k, n, b)|.\]
\end{thm}

This essentially resolves \cref{conj:main}, except in the case where $k$ is small. Note that \cref{thm:frankl} for $k \ge 6$ can be deduced as a special case of \cref{thm:main}. The remainder of this paper will be dedicated to proving \cref{thm:main}. We first briefly describe our method of proof. Central to our method is the technique of \textit{shifting}, first introduced by Erd\H{o}s, Ko, and Rado \cite{erdos-ko-rado}.
We will define shifting in \cref{sec:shifting}.
Essentially, shifting is an operation that can be performed on a family of sets and introduces structure into the family while preserving the absence of a perfect matching. We first show in \cref{sec:shiftF}, using a method similar to that used by Frankl in proving \cref{thm:frankl}, that the family $\mcF$ can be shifted on all but a constant number of elements of $[n]$ while maintaining the property that it has no small blocking set. Then, in \cref{sec:sizeb}, we exploit the shifted structure of $\mcF$ to show that in order for it to be big enough, it must have a blocking set of size exactly $b$. In \cref{sec:redgraph}, we use the blocking set of size $b$ in addition to the mostly shifted structure of $\mcF$ to reduce proving the theorem to proving \cref{prop:graph}, a much simpler statement which is about ordinary graphs. Finally, we prove \cref{prop:graph} in \cref{sec:pfgraph}.

\section{Definitions and preliminary observations}
We start with some definitions. We will refer to the elements of $[n]$ as \textit{vertices}. As mentioned in the introduction, a \textit{matching} is any collection of disjoint sets, and is said to \textit{cover} a set if all elements of the set are contained in one of the elements in the matching. We will frequently abuse notation by saying that a collection of sets contains a vertex $x$ to mean that $x$ belongs to one of the sets of the collection. Also, if $\mcF$ is a family of sets, we define $\mcF(x)$ to be the family of all sets $F$ such that $x \notin F$ and $F \cup \{x\} \in \mcF$.

We use standard asymptotic notation throughout, including big $O$ and little $o$, big $\Omega$ and little $\omega$, and the relation $\sim$ (we say that $f \sim g$ if $f = (1+o(1))g$). All such asymptotic notation will be with respect to $n$ (or equivalently $s$); in other words, $k$ and $b$ are considered to be constant for all asymptotic notation. We will also assume throughout that $n$ is at least a sufficiently large function of $k$ and $b$, using this assumption implicitly on many occasions.

\subsection{Size of \texorpdfstring{$\mc E(k, n, b)$}{E(k, n, b)}} \label{sec:sizeE}

We now approximate the size of $\mc E(k, n, b)$, which will be useful for later comparisons. First, we show an approximation for the binomial coefficient $\binom{n-t}{k-1}$. We have, for any constant $t$, that
\[\binom{n-t}{k-1} = \binom{n}{k-1} - \sum_{i=1}^t \binom{n-i}{k-2} = \binom{n}{k-1} - t\binom{n}{k-2} + O(n^{k-3}).\]
Now, note that the size of each $\mc{E}_i$ (defined in \cref{ex:opt}) is
\[|\mc{E}_i| = \binom{n-b}{k-1} - \binom{n-b-(k-1)(i-1)}{k-1} + 1.\]
(In the case that $i=b$, the final $1$ term should be omitted.) Then we have
\begin{align}
|\mc E(k, n, b)| &= \binom{n-b}{k} + b \binom{n-b}{k-1} - \sum_{i=1}^{b} \binom{n-b-(k-1)(i-1)}{k-1} + b-1 \nonumber \\
&= \binom{n}{k} - \sum_{i=1}^b \binom{n-i}{k-1} + b \binom{n-b}{k-1} - \sum_{i=1}^{b} \binom{n-b-(k-1)(i-1)}{k-1} + b-1 \nonumber \\
&= \binom{n}{k} - b\binom{n}{k-1} + \sum_{i=1}^b i\binom{n}{k-2} - b^2\binom{n}{k-2} + \sum_{i=1}^b (b+(k-1)(i-1))\binom{n}{k-2} + O(n^{k-3}) \nonumber \\
&= \binom{n}{k} - b\binom{n}{k-1} + \left(\frac{b(b+1)}{2} - b^2 + b^2 + \frac{b(b-1)}{2}(k-1)\right)\binom{n}{k-2} + O(n^{k-3}) \nonumber \\
&= \binom{n}{k} - b\binom{n}{k-1} + \left(b + \frac{b(b-1)k}{2}\right)\binom{n}{k-2} + O(n^{k-3}). \label{eq:sizeE}
\end{align}
Thus, $\mc E(k, n, b)$ is missing
\begin{equation}
b\binom{n}{k-1} - \left(b + \frac{b(b-1)k}{2}\right)\binom{n}{k-2} + O(n^{k-3}) \label{eq:missingE}
\end{equation}
sets.

We are now able to check the claim in \cref{ex:k3} that when $k=3$, then $\mc E'(3, n, b)$ is actually larger than $\mc E(3, n, b)$. When $k=3$ we have
\begin{align*}
|\mc E(3, n, b)| &= \binom{n-b}{3} + b \binom{n-b}{2} - \sum_{i=1}^{b} \binom{n-b-2(i-1)}{2} + b-1 \\
&= \binom{n-b}{3} + b \binom{n-b}{2} - \sum_{i=1}^{b} \frac{(n-b)(n-b-1) - 2(2n-2b-1)(i-1) + 4(i-1)^2}{2} + b-1 \\
&= \binom{n-b}{3} + b \binom{n-b}{2} - \frac{b(n-b)(n-b-1)}{2} + \frac{(2n-2b-1)(b-1)b}{2} - \frac{(b-1)b(2b-1)}{3} + b-1 \\
&= \binom{n-b}{3} + b \binom{n-b}{2} - \frac{1}{2}bn^2 + \frac{1}{2}(4b^2 - b)n + \frac{1}{6}(-13b^3 + 6b^2 + 7b - 6).
\end{align*}
We can also compute
\begin{align*}
|\mc E'(3, n, b)| &= \binom{n-b}{3} + b \binom{n-b}{2} - b\binom{n-2b+1}{2} \\
&= \binom{n-b}{3} + b \binom{n-b}{2} - \frac{1}{2}bn^2 + \frac{1}{2}(4b^2 - b)n + (-2b^3 + b^2).
\end{align*}
Thus, the difference in sizes is 
\[|\mc E'(3, n, b)| - |\mc E(3, n, b)| = \frac{1}{6}(b^3-7b+6),\]
which is strictly positive whenever $b > 2$.

\subsection{Shifting} \label{sec:shifting}
The operation of shifting was originally defined by Erd\H{o}s, Ko, and Rado \cite{erdos-ko-rado},
and has been useful for many different problems related to the intersections of families of sets. We refer the reader to Frankl and Tokushige's survey \cite{frankl-tokushige}, as well as Frankl's older survey \cite{frankl-shifting}, for an exposition of the use of shifting on various problems.

For a family $\mcF$ and vertices $x, y$ such that $|\mcF(x)| \ge |\mcF(y)|$, we will define the \textit{shift operator} $S_{x, y}$ on $\mcF$. For each $F \in \mcF$ define the shift of $F$ as follows:
\[
S_{x, y}(F) = 
\begin{cases}
(F \setminus \{y\}) \cup \{x\}, & \text{if $y \in F$, $x \notin F$, and $(F \setminus \{y\}) \cup \{x\} \notin \mcF$,} \\
F, &\text{else}
\end{cases}
\]
Then, the shift of $\mcF$ is defined as
\[S_{x, y}(\mcF) = \{S_{x, y}F: F \in \mcF\}.\]
Intuitively, the shift can be thought of as changing $y$ to $x$ in all sets of $\mcF$ such that doing so would not create a collision with another set already in $\mcF$. We say that the shift is \textit{meaningful} if $S_{x, y}(\mcF) \neq \mcF$.

Shifting is a useful operation primarily due to two properties: that it never creates a perfect matching and that it can only be performed finitely many times. Specifically, we have the following two facts, whose proofs can both be found in \cite{frankl-shifting}.
\begin{fact}
If $\mcF$ has no matching of size $r$, then $S_{x, y}(\mcF)$ also has no matching of size $r$. In particular, letting $r=s$, shifting preserves the property of having no perfect matching.
\end{fact}
\begin{fact}
If $\mcF' = S_{x,y}(\mcF)$ is a meaningful shift of $\mcF$, then
\[\sum_{i=1}^n |\mcF'(i)|^2 > \sum_{i=1}^n |\mcF(i)|^2.\]
\end{fact}
From the above monovariant, we deduce that only a finite number of meaningful shifts can be performed on any family (even if the vertices of $[n]$ are permuted in between shifts). 

Now, we say a family $\mcF$ is \textit{shifted} on a set of vertices $Y$ if $|\mcF(x)|$ is decreasing in $x$ for all $x \in Y$, and there are no meaningful shifts that can be performed on $\mcF$ which use two vertices in $Y$. Shifted families are useful because they have a lot of structure, due to the following fact which follows directly from the definition of a shifted family.

\begin{fact} \label{fact:shifted}
Suppose $\mcF$ is shifted on $Y$. Then, if $F \in \mcF$ contains $y$ but not $x$, where $x < y$ and $x, y \in Y$, then $(F \setminus \{y\}) \cup \{x\} \in \mcF$. 
\end{fact}

Eventually, we will be able to show that any large enough $\mcF$ satisfying the conditions of \cref{thm:main} can be made to be shifted on a large number of vertices, and we will then be able to use this structure to deduce a lot about $\mcF$.

\subsection{Blocking sets have no covering matching}
Before beginning the proof of \cref{thm:main}, we prove one last lemma. The following lemma states that if $\mc F$ is large enough, then a small blocking set of $\mcF$ cannot be covered by any matching (even a matching with size smaller than the size of the blocking set). This fact that will be useful several times.

\begin{lemma} \label{lem:coveringmatching}
Suppose that $|\mc F| > |\mc E(k, n, b)|$ and let $B$ be a blocking set of $\mcF$, where $|B| \le s/2$. Then there is no matching that covers $B$.
\end{lemma}
\begin{proof}
Suppose otherwise, and let the matching $M$ cover $B$. Furthermore let $M$ be of maximal size; note that its size must be less than $|B|$ since $B$ is a blocking set. Then $\mcF$ cannot contain a set disjoint from $M$, since otherwise we could append that set to $M$ to get a larger matching covering $B$. $M$ contains at most $k|B| \le n/2$ vertices, so the number of sets missing from $\mcF$ is at least
\[\binom{n/2}{k} = \Omega(n^k),\]
so $\mcF$ is missing more sets than $\mc E(k, n, b)$ (since the number of missing sets of $\mc E(k, n, b)$ is given by \eqref{eq:missingE}), a contradiction.
\end{proof}

\section{Making \texorpdfstring{$\mcF$}{F} shifted on most of \texorpdfstring{$[n]$}{[n]}} \label{sec:shiftF} 
We will now begin with the proof of \cref{thm:main}. We proceed by induction on $b$. The base case $b=1$ is Kleitman's result in \cite{kleitman}.
Assume true for all smaller values of $b$, and also suppose for the sake of contradiction that there exists a family $\mcF$ of $k$-subsets of $[n]$ with no perfect matching and no blocking set of size less than $b$ such that $|\mcF| > |\mc E(k, n, b)|$. Further assume that $\mcF$ is of maximum possible size.

In this section we will prove the following proposition, which implies that we can assume that $\mcF$ is shifted on all but a constant number of vertices. Our proof of this proposition is based on the proof of Proposition 4.4 in \cite{frankl-main}.

\begin{prop} \label{prop:shift}
There exists a constant $c = c(k, b)$ such that it is possible to apply shifts to $\mcF$ and permute the vertices in $[n]$ so that $\mcF$ has no blocking set of size less than $b$ and is shifted on $[n-c]$.
\end{prop}
\begin{proof}
Suppose that $\mcF$ has been shifted as much as possible while maintaining the property that $\mcF$ has no blocking set of size less than $b$. Thus, every meaningful shift of $\mcF$ creates a blocking set of size less than $b$. 

Permute the vertices, sorting by the size of $\mcF(i)$, so that $\mcF(1)$ has the largest size. Then, let $A = [n-2b, n]$ be the last $2b+1$ vertices in $[n]$. Note that no $b'$-subset of $A$, for $b' < b$, can form a blocking set, so there must exist matchings of $b'$ sets that cover each such subset. Let $T$ consist of all the vertices in $[n]$ which are covered by any such matching. We then have $|T| \le 2^{2b+1}(b-1)k$. 

We now claim that $\mcF$ is shifted on $[n] \setminus T$. Suppose otherwise, so that there is a meaningful shift $S_{x, y}$, where $x < y$ and $x, y \notin T$. By assumption, this means that the shifted family $\mcF' = S_{x, y}(\mcF)$ must contain a blocking set $B$ of size $b'$, for some $b' < b$. We may furthermore pick $B$ so that $b'$ is minimal. Note that the matching which covers any $b'$-subset of $A$ is contained entirely in $T$, and thus does not contain $y$, and is therefore not affected by the shift $S_{x, y}$. Thus, every $b'$-subset of $A$ is still not a blocking set of $\mc F'$, so $B$ must contain an vertex not in $A$. We show that every vertex in $B$ is in a small number of sets of $\mcF'$. 

\begin{lemma}
For every $z \in B$, $|\mcF'(z)| = O(n^{k-2})$.
\end{lemma}
\begin{proof}

Note that by \cref{lem:coveringmatching}, there is no matching that covers $B$.
Thus we may add all $k$-sets disjoint from $B$ to $\mcF'$, and $B$ will still be a blocking set of minimal size. By the inductive hypothesis, after adding these sets, the size of $\mcF$ will be at most
\[\binom{n}{k} - b' \binom{n}{k-1} + O(n^{k-2}),\]
by \eqref{eq:sizeE}. Then, since none of the sets which are disjoint from $B$ contain $z$, we have
\[|\mcF'(z)| \le \binom{n}{k} - b' \binom{n}{k-1} + O(n^{k-2}) - \binom{n-b'}{k} = O(n^{k-2}),\]
as desired.
\end{proof}
Thus, since $B$ contains a vertex not in $A$, we have some vertex $z \in [n-2b-1]$ so that $|\mcF'(z)| = O(n^{k-2})$. If $z \neq y$, then \[|\mcF(z)| \le |\mcF'(z)| = O(n^{k-2}).\] If on the other hand $z = y$, then
\begin{align*}
|\mcF(x)| + |\mcF(y)| = |\mcF'(x)| + |\mcF'(y)| \le \binom{n}{k-1} + O(n^{k-2}),
\end{align*}
so either $|\mcF(x)|$ or $|\mcF(y)|$ is at most $\frac 12 \binom{n}{k-1} + O(n^{k-2})$. In any case, we have some $z' \in [n-2b-1]$ such that $|\mcF(z')| \le \frac 12 \binom{n}{k-1} + O(n^{k-2})$. By monotonicity of $|\mcF(i)|$, we also have that $|\mcF(w)| < |\mcF(z')|$ for all $w \in A$. Thus, the total number of sets in $\mcF$ which intersect $A$ is at most
\[\sum_{w \in A} |\mcF(w)| \le \frac{2b+1}{2} \binom{n}{k-1} + O(n^{k-2}).\]
Then, since there are at most $\binom{n-2b-1}{k}$ sets not intersecting $A$, the total size of $\mcF$ is at most
\[\binom{n-2b-1}{k} + \frac{2b+1}{2} \binom{n}{k-1} + O(n^{k-2}) = \binom{n}{k} - \frac{2b+1}{2} \binom{n}{k-1} + O(n^{k-2}) < |\mc E(k, n, b)|.\]

This is a contradiction, so we must indeed have that $\mcF$ is shifted on $[n] \setminus T$. Permuting the vertices so that $T$ is at the end of $[n]$, we have that $\mcF$ is shifted on $[n-c]$, as desired, concluding the proof of \cref{prop:shift}.
\end{proof}

We henceforth assume that $\mcF$ is shifted on $[n-c]$.

\section{\texorpdfstring{$\mcF$}{F} has a blocking set of size exactly \texorpdfstring{$b$}{b}} \label{sec:sizeb}
In this section we will show that $\mcF$ has a blocking set of size exactly $b$. Suppose for a contradiction otherwise, and let $B$ be the smallest blocking set of $\mcF$, where $|B| = b' > b$. We will split into cases based on whether $B$ is small or large.

\subsection{\texorpdfstring{$B$}{B} small}
Here assume that $b' \le s/2 = n/2k$. 

\begin{claim}
$\mcF$ contains a set of the form $\{z\} \cup C$, where $z \in B$ and $C \subset [b'k+1, n-c]$.
\end{claim}
\begin{proof}
For any fixed $z \in B$, the number of possible such sets $\{z\} \cup C$ is at least the number of such sets which contain no vertex in $B$ other than $z$. Thus the total number of possible $\{z\} \cup C$ is at least
\begin{equation} \label{eq:numgoodsets}
b' \binom{n-c-b'k - b'}{k-1} \sim \frac{b' (n-(k+1)b')^{k-1}}{(k-1)!},
\end{equation}
where we have used that $n-(k+1)b' \ge (1-(k+1)/2k)n = \omega(1)$.

Now, if $b' \le \sqrt n$, then the quantity in \eqref{eq:numgoodsets} is at least \[(1+o(1)) (b+1) n^{k-1}/(k-1)!\] (since $b' > b$). If on the other hand $b' > \sqrt n$, then the same quantity is at least
\[(1+o(1)) \frac{\sqrt n ((1-(k+1)/2k)n)^{k-1}}{(k-1)!} = \Omega(n^{k-1/2}).\]
In either case, the number of such sets $\{z\} \cup C$ is greater than the number of missing sets of $\mcF$, which is at most the quantity \eqref{eq:missingE}. Thus, $\mcF$ must contain some such set $\{z\} \cup C$.
\end{proof}

Consider a matching $M$ with $b'-1$ sets covering $B \setminus \{z\}$ (this exists since $B \setminus \{z\}$ is not a blocking set). By \cref{lem:coveringmatching}, $M$ cannot contain $z$. Then this matching covers at most $(b'-1)k$ vertices, so it is possible by repeated application of \cref{fact:shifted} to shift $\{z\} \cup C$ to a set $\{z\} \cup C' \in \mcF$ so that $C'$ is disjoint from $M$. But then we can add $\{z\} \cup C'$ to $M$ to get a matching covering $B$, a contradiction to \cref{lem:coveringmatching}.

% Now, if $\mcF$ contains any set disjoint from $M$, then this set adjoined to $M$ would create a matching of size $b'$ covering $B$, which is impossible since $B$ is a blocking set. Thus, $\mcF$ contains no set disjoint from $M$, so the number of missing sets of $\mcF$ is at least
% \[\binom{n-(b'-1)k}{k} \ge \binom{n/2}{k} = \Omega(n^k),\]
% a contradiction.

Thus we conclude that the smallest blocking set $B$ cannot have size $b'$, where $b < b' \le s/2$.

\subsection{\texorpdfstring{$B$}{B} large} \label{sec:redgraph}
Assume now that $b' > s/2$. We will then find a perfect matching in $\mcF$ in order to derive a contradiction.

\begin{lemma} \label{lem:allsets}
$\mcF$ contains every $k$-vertex subset of $[n-b'+1]$.
\end{lemma}
\begin{proof}
By repeated application of \cref{fact:shifted}, it is enough to show that $\mcF$ contains some subset of $[n-b'+2, n-c]$. But if this is not the case, then the number of missing sets of $\mcF$ is at least
\[\binom{b'-c-1}{k} \ge \binom{s/2-c-1}{k} = \Omega(n^k),\]
a contradiction.
\end{proof}

Now, since there are no blocking sets of size smaller than $b'$, there is a matching $M$ in $\mcF$ which covers $[n-b'+2, n]$. By \cref{lem:allsets}, we can then add an arbitrary matching on the vertices not covered by $M$ to obtain a perfect matching in $\mcF$, a contradiction.

We may thus conclude that $B$ also cannot have size $b' > s/2$. Therefore, combining the results from the two cases, the smallest blocking set of $\mcF$ has size exactly $b$.

\section{Reduction to graph problem}
Permute the vertices in $[n]$ again, so that $[b]$ is a blocking set of $\mcF$ and $\mcF$ is shifted on $[b+c+1, n]$. We show the following fact which allows us to ignore most vertices in $[n]$:

\begin{fact} \label{fact:nocoveringG}
Suppose $M$ is a collection of at most $b$ $k$-sets in $\mcF$ which covers $[b]$ and such that no two sets in $M$ both contain the same vertex in $[b+c+bk]$. Then, there exists a matching in $\mcF$ which covers $[b]$.
\end{fact}
\begin{proof}
For $F \in M$, write $F = F_0 \sqcup F_1$, where $F_0 = F \cap [b+c+bk]$ and $F_1 = F \cap [b+c+bk+1, n]$. Then, by repeated application of \cref{fact:shifted}, we also have $F_0 \sqcup F_1' \in \mcF$ for any $F_1' \subset [b+c+1, b+c+bk]$ which has the same size as $F_1$ and is disjoint from $F_0$. We can furthermore pick $F_1'$ to be disjoint from all of $M$, since $M$ contains at most $bk$ vertices in total. Thus, we can replace $F$ with $F_0 \sqcup F_1'$ in $M$. Doing this successively to all $F \in M$, we get that $M$ becomes a matching. Note that all vertices in $[b]$ were in $F_0$ and not in $F_1$ at each of these steps, so $M$ still covers $[b]$.
\end{proof}

Now, define $\mc G$ to be the set of all sets $S$ which contain a vertex in the blocking set $[b]$ and such that there exists a set $T \subset [b+c+bk+1, n]$ satisfying $S \cup T \in \mc F$. (Note that all sets in $\mc F$ are automatically in $\mc G$.) By \cref{fact:nocoveringG}, if there is a matching in $\mc G$ which covers $[b]$, then there is also a matching in $\mc F$ which covers $[b]$, and there can be no such matching by \cref{lem:coveringmatching}. (Though $\mc G$ contains sets of different sizes, a matching is similarly defined to be a subcollection of disjoint sets.) Thus, there is no matching in $\mc G$ which covers $[b]$.

Also, note that $\mc G$ contains no single-element sets: if it contained the single-element set $\{z\}$, then there would be no matching in $\mc G$ (and thus no matching in $\mc F$) covering $[b] \setminus\{z\}$, and thus $[b] \setminus\{z\}$ would be a blocking set of size $b-1$, a contradiction.

Now, any set of size $r$ in $\mc G$ corresponds to roughly $\binom{n}{k-r} = \Theta(n^{k-r})$ possible sets in $\mc F$. Since $\mc G$ has no $1$-sets, the highest-order contribution comes from the 2-sets in $\mc G$. Thus we will be interested in bounding the number of 2-sets in $\mc G$.
To this end, define the graph $G$ on the vertices $[n]$ with an edge between any two vertices whenever the set containing those two vertices is in $\mc G$. (Note that every edge in $G$ touches $[b]$.) Without loss of generality (by permuting $[b]$) assume that the vertex $1$ has the highest degree in $G$ out of any vertex in $[b]$.

Now, $\mcF$ must contain a matching which covers $[2, b]$ (and which contains only sets touching $[2, b]$) since it is not a blocking set. But every set in $\mcF$ which intersects $[b]$ is also a set in $\mc G$. Thus, there exists a matching $S_1, \dots, S_a$ of sets in $\mc G$ which covers $[2, b]$ (note that it cannot contain the vertex $1$). We will now use the following proposition about the graph $G$.

\begin{prop} \label{prop:graph}
Suppose that either $k \ge \max \{b+1, 6\}$ or $k \ge 100$. Let $G$ be a (simple, undirected) graph on $[n]$ such that all edges of $G$ touch $[b]$, and $1$ has the greatest degree out of any vertex in $[b]$. Let there also be a collection of mutually disjoint $k$-sets $\alpha_1, \dots, \alpha_a \subset [n]$ which each contain some vertex in $[b]$ and whose union contains all vertices in $[2, b]$. Suppose that it is impossible to cover $[b]$ by a matching which consists only of edges of $G$ and some of the $\alpha_i$. Then, the total number of edges in $G$ is at most $\frac{b(b-1)}{2}(k-1)$.

Moreover, equality holds only in the following case: $a = b-1$, and the $k$-sets $\alpha_1, \dots, \alpha_a$ each contain only one element of $[2, b]$. Up to permutation of the vertices in $[2, b]$, there is an edge from $i$ to every element of the $k$-set that contains $j$, except $j$ itself, for every $1 \le i < j \le b$. These are the only edges in $G$. This equality case is illustrated in \cref{fig:eqcase}.
\end{prop}

\begin{figure}[ht]
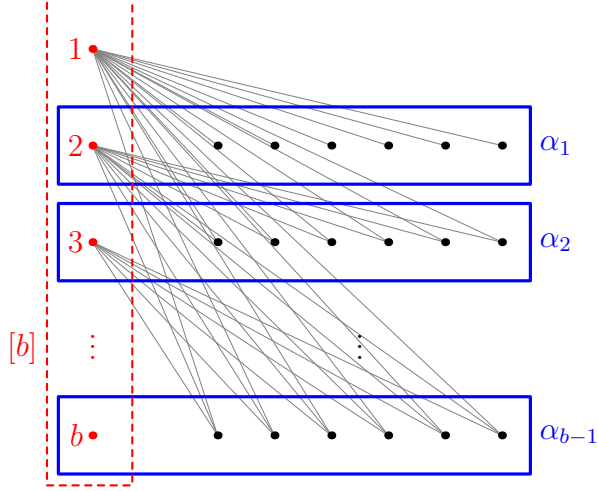

\centering
\begin{asy}
size(0, 6.5cm);
real r = 1.7; // y scale

pen edgepen = grey + 0.3;
for (int y1=2; y1<=4; ++y1) {
	for (int y2=0; y2<y1; ++y2) {
		if (y2==1) continue;
		for (int x=2; x<=7; ++x) {
			draw((-0.2, r*y1)--(x, r*y2), edgepen);
		}
	}
}

for (int x=2; x<=7; ++x) {
	for (int y=0; y<=3; ++y) {
		if (y != 1) { dot( (x,r*y) ); }
	}
}
dot("$1$", (-0.2,4*r), dir(180), red);
dot("$2$", (-0.2,3*r), dir(180), red);
dot("$3$", (-0.2,2*r), dir(180), red);
label("$\vdots$", (-0.2,r), red);
dot("$b$", (-0.2,0), dir(180), red);

label("$\vdots$", (4.5,r));

pen alphapen = blue+1.2;
draw(box((-0.8,3.4*r), (7.5,2.6*r)), alphapen);
draw(box((-0.8,2.4*r), (7.5,1.6*r)), alphapen);
draw(box((-0.8,0.4*r), (7.5,-0.4*r)), alphapen);
label("$\alpha_1$", (7.5,3*r), dir(0), alphapen);
label("$\alpha_2$", (7.5,2*r), dir(0), alphapen);
label("$\alpha_{b-1}$", (7.5,0), dir(0), alphapen);

draw(box((-1,-0.4*r-0.2), (0.5, 4.4*r+0.2)), red+linetype("3 3")+0.8);
label("$[b]$", (-1, 0.9*r), dir(180), red);
\end{asy}
\caption{The only equality case of \cref{prop:graph}, up to permutation of the vertices in $[2, b]$.}
\label{fig:eqcase}
\end{figure}

We will prove \cref{prop:graph} in the next section, but first we will show how it can be used to deduce \cref{thm:main}. Apply \cref{prop:graph} directly to $G$ with $\alpha_i = S_i$; we get that $G$ must have at least $\frac{b(b-1)}{2}(k-1)$ edges. We claim equality must hold; for a contradiction suppose $G$ has strictly fewer edges.

First note that $\mc F$ is missing every set of the form $\{z\} \cup T$, where $z \in [b]$ and $T \subset [b+c+bk+1, n]$. The number of such missing sets is 
\[b\binom{n-b-c-bk}{k-1}.\]
Furthermore, $\mc F$ is also missing every set of the form $\{z, w\} \cup T$, where $z \in [b]$, $w \in [b+c+bk]$, $\{z, w\}$ is not an edge in $G$, and $T \subset [b+c+bk+1, n]$. The number of such edges $\{z, w\}$ is
\[\frac{b(b-1)}{2} + b(c+bk) - |E(G)| \ge b(c+bk) - \frac{b(b-1)}{2}(k-2) + 1.\]
This causes the following additional number of missing sets of $\mc F$:
\[\left( b(c+bk) - \frac{b(b-1)}{2}(k-2) + 1 \right) \binom{n-b-c-bk}{k-2}.\]
Therefore, the total number of missing sets of $\mc F$ is at least
\begin{align*}
&\phantom{=}b\binom{n-b-c-bk}{k-1} + \left( b(c+bk) - \frac{b(b-1)}{2}(k-2) + 1 \right) \binom{n-b-c-bk}{k-2} \\
&= b\left(\binom{n}{k-1} - (b+c+bk)\binom{n}{k-2} + O(n^{k-3})\right) \\ &\qquad \qquad + \left( b(c+bk) - \frac{b(b-1)}{2}(k-2) + 1 \right) \binom{n}{k-2} + O(n^{k-3}) \\
&= b\binom{n}{k-1} - \left(b + \frac{b(b-1)k}{2} - 1\right)\binom{n}{k-2} + O(n^{k-3}),
\end{align*}
which is more than the number of missing sets of $\mc E(k, n, b)$ (as seen in \eqref{eq:missingE}), a contradiction.

Therefore, equality must hold in \cref{prop:graph}, so $\mc G$ must contain $2$- and $k$-sets of the prescribed form. 
Now, note that for any $2$-set $E \in \mc G$, we can add $E \cup S$ to $\mcF$ for any set $S$ disjoint from $[b]$ and $E$ and such that $|S| = k-2$. Recall that $\mc G$ has no matching which covers $[b]$ and $\mc G$ also contains all sets in $\mc F$; thus, $[b]$ is still a blocking set of $\mc F$ after these sets are added. Similarly we can add all $k$-subsets of $[b+1, n]$ to $\mc F$. Finally note that $\mc F$ contains the sets $\alpha_1, \dots, \alpha_{n-1}$. But now, note that $\mc F$ contains all the sets in $\mc E(k, n, b)$, where in the construction of $\mc E(k, n, b)$ in \cref{ex:opt}, the sets $E_i$ are chosen to be $E_i = \alpha_i \setminus \{i\}$.
But then by the maximality of $\mc E(k, n, b)$, we have that $\mcF$ cannot contain any more sets, so $\mc F = \mc E(k, n, b)$, concluding the proof of \cref{thm:main}.

\section{Proof of graph problem} \label{sec:pfgraph}
We will now prove \cref{prop:graph}. We proceed by induction on $b$. The cases $b = 1, 2$ are easy to check, so let $b \ge 3$ and assume \cref{prop:graph} holds for all smaller values of $b$. Suppose that $G$ satisfies the conditions of \cref{prop:graph} and has at least $\frac{b(b-1)}{2}(k-1)$ edges. 

First we introduce some notation. We call the vertices in $[b]$ \textit{special}. We refer to the $k$-sets $\alpha_1, \dots, \alpha_a$ as \textit{cells} (so that there are $a$ cells), and if a cell contains two or more special vertices we say that the special vertices it contains are \textit{siblings}. On the other hand, if a cell contains only one special vertex, we say that special vertex is \textit{lonely}. Also, define the \textit{exterior}, $R$, to be the set of vertices which are not $1$ and are also not in any set $\alpha_i$.

We also define the edge count $e(A, B)$, where $A$ and $B$ are sets of vertices, to be the number of edges between $A$ and $B$. We will sometimes abuse notation by writing a single vertex as an argument of $e$ instead of the set containing only that vertex.

Also, we note that $G$ has no edge connecting the vertex $1$ to any vertex which is not in some $\alpha_i$ -- otherwise that edge would form a covering matching of $[b]$ along with all the sets $\alpha_1, \dots, \alpha_a$. Also, there is no edge connecting $1$ to any lonely vertex, since otherwise that edge would form a covering matching along with the sets $\alpha_i$ which do not contain that vertex.

We now make use of the following lemma of graph theory, which is a generalization of Hall's marriage theorem.

\begin{lemma}
Let $G$ be a graph, and let $A$ be a set of vertices of $G$. Suppose that $G$ has no matching which covers $A$. Then, there exists a set $S \subset A$ such that $|N(S)| \le 2|S|-1$. (Here $N(S)$ is the set of neighbors of $S$, including all points of $S$ itself.)
\end{lemma}
\begin{proof}
Let $M$ be a matching in $G$ which contains the maximum possible number of vertices in $A$; we may assume that $M$ contains no edges disjoint from $A$. By assumption, $G$ contains a vertex $v$ which is not in $M$.

We construct $S$ in steps, along with a set $T$ which will eventually contain all neighbors of $S$, except those in $S$ itself. Initially, let $S$ contain only the vertex $v$. While there exists any edge from an element of $S$ to another vertex $u$ which is not in $S$ or $T$, if $u$ has a neighbor in $M$ which is also in $A$, add $u$ to $T$ and the neighbor to $S$. (Note that whenever we add one vertex each to $S$ and $T$, they are adjacent in $M$. Thus, the neighbor of $u$ in $M$ cannot already be in $S$ or $T$.)

Now, suppose that after this process has ended, there is still an edge from a vertex $w \in S$ to a vertex $u \notin S, T$. Note that $u$ cannot be adjacent in $M$ to any vertex in $A$, otherwise it would have previously been added to $T$. Now by construction, there is an even-length path connecting $v$ to $w$, where every other vertex (including $v$ and $w$) is in $S$, and the rest are in $T$, and every other edge is an edge of $M$ (so that the last edge touching $w$ is in $M$, but the edge touching $v$ is not). Then, we can flip the parity of the edges on the path which are in $M$, so that the edge touching $v$ is now in $M$, but the edge touching $w$ is not. Now, $M$ contains $v$ but does not contain $w$, but otherwise contains the same vertices. Finally, add the edge $\{w, u\}$ to $M$, and remove any edge touching $u$ in $M$; since an edge touching $u$ in $M$ cannot contain any vertex in $A$, this does not remove any vertex of $A$ from $M$. Thus, $M$ now covers all the vertices in $A$ which it originally covered, in addition to $v$, contradicting its original maximality.

Thus, there is no edge from any vertex in $S$ to a vertex not in $S$ or $T$. Therefore, $N(S) = S \cup T$. Since one vertex is added to $S$ for each vertex added to $T$, we have $|T| = |S| - 1$, so $|N(S)| = 2|S| - 1$ as desired. 
\end{proof}

Using this lemma on $G$ with $A = [b]$, we can find a set $S \subset [b]$ such that $|S| = d$ and $|N(S)| \le 2d - 1$. If $1 \in S$, then since the vertex $1$ has the maximal degree out of $[b]$, this means that there are at most \[b(2d-1) \le b(2b-1) \le \frac{5b(b-1)}{2} \le \frac{b(b-1)}{2}(k-1)\] edges in $G$. Equality cannot hold since that would require $b=d$, and then we would be double-counting the edge between the vertices $1$ and $2$. Thus we are immediately done if $1 \in S$. We will assume henceforth that $1 \notin S$.

Now let $\Gamma$ denote the set of cells $\alpha_i$ which contain some vertex in $S$, and let $|\Gamma| = d'$. Clearly we have $d' \le d$. Also, let $T$ denote the set of special vertices which are siblings of some vertex in $S$. Then, $S \cup T$ is the set of all special vertices in the cells of $\Gamma$. Let $|T| = \ell$. Each cell $\alpha_i$ contains a vertex which is not in $T$ (and which is also not 1), so we have 
\begin{equation} \label{eq:bda}
a \le b - \ell - 1.
\end{equation}

Now, we wish to induct by removing the cells in $\Gamma$. The idea is that since $S$ has few neighbors and therefore low total degree, we will not be removing too many edges. After removing the cells in $\Gamma$ and all edges connected to vertices in these cells, there will still be no matching that covers the remaining special vertices, since otherwise that matching, along with the cells of $\Gamma$, would cover $[b]$. Thus, as long as the vertex $1$ still has the highest degree, we will be able to apply the inductive hypothesis.

We delete the edges touching $\Gamma$ in a specific order that will allow us to more easily bound the number of deleted edges. Specifically, we delete the edges in four steps as follows.

\begin{enumerate}
\item Delete all edges touching $S$. Let $X$ be the set of deleted edges. Since $|N(S)| \le 2d - 1$, we have 
\begin{equation} \label{eq:bdX}
|X| \le d(d-1) + \binom{d}{2} = \frac{3d(d-1)}{2}.
\end{equation}

\item Delete all remaining edges from $[b] \setminus (S \cup T)$ to $\Gamma$. Note that this deletes at most $d'k - d$ edges from each vertex in $[b] \setminus (S \cup T)$.

\item If vertex $1$ is no longer the maximum-degree vertex in $[b] \setminus (S \cup T)$, then arbitrarily delete edges from any such vertices with higher degree (except those touching $1$) until $1$ is again the maximum-degree vertex in $[b] \setminus (S \cup T)$. Let $Y$ be the set of edges deleted in steps 2 and 3 combined. 

Let $p$ be the total number of elements of $S$ which are adjacent to the vertex $1$. Then, the number of edges touching $1$ which were deleted in steps 1 and 2 combined is at most $d'k - d + p$. This means that $Y$ contains at most $d'k - d + p$ edges touching each vertex in $[b] \setminus (S \cup T)$. Thus,
\begin{equation} \label{eq:bdY1}
|Y| \le (b-d-\ell)(d'k - d + p).
\end{equation}

We will now bound $p$. Note that a cell $\gamma \in \Gamma$ which contains $t$ vertices of $S$ contributes at most $t$ to $p$, but only if $\gamma$ contains a sibling, since $1$ is not adjacent to any lonely vertices. If $t \ge 2$, then $\gamma$ also contributes at least $t - 1 \ge t/2$ to the quantity $d - d'$. Otherwise, $t = 1$ and $\gamma$ must contain a vertex not in $S$, in which case $\gamma$ contributes 1 to $\ell$. Therefore, we have the inequality
\[p \le 2(d-d') + \ell.\]
Thus, the above bound on $Y$ becomes
\begin{align}
|Y| &\le (b-d-\ell)(d'k - d + 2(d-d') + \ell) \nonumber \\
&= (b-d-\ell)(d'(k-2)+d) + \ell(b-d-\ell) \nonumber \\
&\le (b-d-\ell)d(k-1) + \ell(b-d-\ell). \label{eq:bdY2}
\end{align}

\item Delete all remaining edges touching $T$. Note that since $1$ was the vertex of maximum degree, this deletes at most $\ell \deg(1)$ edges (where the degree is calculated before any edges were deleted). Let $r$ be the total number of cells $\alpha_i$ which contain siblings. Since $1$ does not have an edge to any lonely vertex, it has at most $k-1$ edges to each cell except for these $r$ cells, so
\[\deg(1) \le a(k-1) + r.\]
Now, each of the $r$ cells which contain siblings contain at least $2$ special vertices (while all the others still contain 1), so since $b-1$ is the number of special vertices which are contained in some cell, this means that $b-1 \ge 2r + (a - r)$. Rearranging, $r \le b - a - 1$. Therefore, using this and \eqref{eq:bda}, the above bound then becomes
\begin{align}
\deg(1) &\le a(k-1) + r \label{eq:bddeg11} \\
&\le a(k-1) + (b-a-1) \nonumber \\
&= a(k-2) + b - 1 \nonumber \\
&\le (b-\ell-1)(k-2) + b - 1 \label{eq:bddeg14} \\
&= (b-\ell)(k-1) + \ell - (k-1). \nonumber
\end{align}
Thus, if $Z$ is the number of edges deleted in this step, then
\begin{equation} \label{eq:bdZ}
|Z| \le \ell \deg(1) \le \ell ((b-\ell)(k-1) + \ell - (k-1)).
\end{equation}
\end{enumerate}

Now, we will apply the inductive hypothesis as described earlier. Note that all vertices touching $\Gamma$ are deleted in steps 1, 2, and 4, and step 3 ensures by deleting some more edges that vertex 1 still has maximal degree. Thus the inductive hypothesis can be applied; since there are $b-d-\ell$ special vertices remaining, the number of edges remaining is at most $\frac{(b-d-\ell)(b-d-\ell-1)}{2}(k-1)$. Then, using \eqref{eq:bdX}, \eqref{eq:bdY2}, and \eqref{eq:bdZ}, the total number of edges originally in $G$ is at most
\begin{align}
&\phantom{=} \frac{(b-d-\ell)(b-d-\ell-1)}{2}(k-1) + |X| + |Y| + |Z| \nonumber \\
&\le \frac{(b-d-\ell)(b-d-\ell-1)}{2}(k-1) + \frac{3d(d-1)}{2} + (b-d-\ell)d(k-1) + \ell(b-d-\ell) \nonumber \\
&\qquad \qquad + \ell ((b-\ell)(k-1) + \ell - (k-1)) \nonumber \\
&\le \frac{(b-d-\ell)(b-d-\ell-1)}{2}(k-1) + \frac{d(d-1)}{2}(k-1) + (b-d-\ell)d(k-1) \nonumber \\
&\qquad \qquad + \ell ((b-\ell)(k-1) + \ell - (k-1) + (b - d - \ell)) \label{eq:a3} \\
&= \frac{(b-\ell)(b-\ell-1)}{2}(k-1) + \ell(b - \ell)(k-1) + \ell(\ell - (k-1) + (b-d-\ell)) \nonumber \\
&= \frac{b(b-1)}{2}(k-1) - \frac{\ell(\ell-1)}{2}(k-1) + \ell(\ell - (k-1) + (b-d-\ell)) \nonumber \\
&\le \frac{b(b-1)}{2}(k-1) - \frac{\ell(\ell-1)}{2}(k-3) + \ell(-(k-2) + (b-d-\ell)) \nonumber \\
&\le \frac{b(b-1)}{2}(k-1) - \frac{\ell(\ell-1)}{2}(k-3) + \ell(-(k-2) + (b-1-1)) \nonumber \\
&= \frac{b(b-1)}{2}(k-1) - \frac{\ell(\ell-1)}{2}(k-3) + \ell(-k+b). \label{eq:a8}
\end{align}

If $k \ge b+1$, then this is at most $\frac{b(b-1)}{2}(k-1)$, the desired bound. To complete the proof of the $k \ge \max\{b+1, 6\}$ case, it remains to check that the equality case is as described in \cref{prop:graph}. 

To this end, suppose equality holds. In order to achieve equality in the very last step, we must have $\ell = 0$. Also, in order to have equality at \eqref{eq:a3}, we also have $d=1$. Without loss of generality suppose that $S = \{b\}$, and that $\alpha_{b-1}$ contains the vertex $b$ (and no other special vertices, since $\ell = 0$). Also, by the inductive hypothesis, the edges not touching $\alpha_{b-1}$ take exactly the form in \cref{prop:graph} (but for $b-1$ instead of $b$). It remains to check that the edges touching $\alpha_{b-1}$ are exactly those that go from $[b-1]$ to vertices in $\alpha_{b-1}$, except the vertex $b$ itself. However, note that 1 has no edge to $b$ (since $b$ is lonely), and any vertex $i$ in $[2, b-1]$ cannot have an edge to $b$ since then one could make a matching consisting of the edge from $i$ to $b$, all the cells except $\alpha_{b-1}$ and the one containing $i$, and any edge from $1$ to the cell containing $i$ (this exists by the inductive hypothesis), and this would form a covering matching of $[b]$. Thus, there are no edges containing $b$.

On the other hand, equality must hold in \eqref{eq:bdY1}. Since $d'=d=1$ and $p=0$ (since all vertices are lonely), this means that in steps 2 and 3, $k-1$ edges were deleted from each vertex in $[b-1]$. Recall that (by the inductive hypothesis) after deletion of edges, every vertex in $[2, b-1]$ has strictly lower degree than $1$. Therefore, no edges could have been deleted in step 3, so $k-1$ edges were deleted from each vertex in $[b-1]$ in step 2. Thus, every vertex in $[b-1]$ must have been connected to every vertex in $\alpha_{b-1}$ except $b$. Finally, note that \eqref{eq:bdX} implies that the vertex $b$ has no edges at all. This completes the analysis of the equality case; we have shown that the only deleted edges are those from $[b-1]$ to $\alpha_{b-1}$ (but not to $b$), and thus the graph $G$ takes exactly the form described in \cref{prop:graph}. This completes the proof of \cref{prop:graph} in the $k \ge \max\{b+1, 6\}$ case.

Now, to prove the $k \ge 100$ case of \cref{prop:graph}, we assume henceforth that $100 \le k \le b$ (since $b < k$ is already covered by the previous case). We bound $|Y|$ in a slightly different way; let $\Phi$ be the set of cells in $\Gamma$ whose vertices are all adjacent to $1$, and let $q = |\Phi|$. (Note that every cell in $\Phi$ must contain siblings, since $1$ is not adjacent to any lonely vertex; we will use this fact later). Then, the number of edges touching 1 which were deleted in steps 1 and 2 is at most $d'(k-1) + q$ (since there are at most $k-1$ edges from 1 to any cell whose vertices aren't all adjacent to it). It is still the case that $d'(k-1)+q$ is at least $d'-d$, which is the maximum number of edges deleted from any vertex in step 2. Thus, as before, there are at most $d'(k-1)+q$ edges deleted from each vertex in $[b] \setminus (S \cup T)$ in steps 2 and 3 combined, so we have the bound
\begin{equation} \label{eq:bdY3}
|Y| \le (b-d-\ell)(d'(k-1) + q).
\end{equation}
Using the bound \eqref{eq:bdY3} instead of \eqref{eq:bdY2}, the bound on the number of edges in $G$ becomes
\begin{align}
&\phantom{=} \frac{(b-d-\ell)(b-d-\ell-1)}{2}(k-1) + |X| + |Y| + |Z| \nonumber \\
&\le \frac{(b-d-\ell)(b-d-\ell-1)}{2}(k-1) + \frac{3d(d-1)}{2} + (b-d-\ell)(d'(k-1) + q) \nonumber \\
&\qquad \qquad + \ell ((b-\ell)(k-1) + \ell - (k-1)) \label{eq:b2} \\
&\le \frac{(b-d-\ell)(b-d-\ell-1)}{2}(k-1) + \frac{d(d-1)}{2}(k-1) + (b-d-\ell)(d(k-1) + q) \nonumber \\
&\qquad \qquad + \ell ((b-\ell)(k-1) + \ell - (k-1)) \label{eq:b3} \\
&= \frac{(b-\ell)(b-\ell-1)}{2}(k-1) + \ell(b - \ell)(k-1) + (b-d-\ell)q + \ell(\ell - (k-1)) \nonumber \\
&= \frac{b(b-1)}{2}(k-1) - \frac{\ell(\ell-1)}{2}(k-3) - \ell(k-1) + (b-d-\ell)q \nonumber \\
&\le \frac{b(b-1)}{2}(k-1) - \frac{\ell(\ell-1)}{2}(k-3) - \ell(k-1) + bq \nonumber \\
&\le \frac{b(b-1)}{2}(k-1) + bq. \label{eq:b7}
\end{align}
Now, if $q = 0$, then again the desired bound is achieved, and it remains to check the equality case. The analysis of the equality case is exactly the same as in the $k \ge \max\{b+1, 6\}$ case, and so we are done if $q = 0$.

Thus, suppose henceforth that $q > 0$. We will show that either the bound \eqref{eq:b7} can be reduced by more than $bq$, or that we can find a matching which covers $[b]$, in both cases obtaining a contradiction. To do the first of these, it is enough that any step of the above chain of inequalities has a difference of more than $bq$.

Now, note that \eqref{eq:b7} is within $bq$ of its assumed lower bound; this means that there is a lot of rigidity in the above chain of inequalities. This means that we will be able to find a lot of structure in $G$, or else one of the steps in the chain of inequalities will be able to be reduced by more than $bq$. To this end, we show some claims about $G$. (All claims that follow are assumed to be stated before any edges in $G$ are deleted.)

% \begin{claim}
% $q \le d < 4b/k < 0.01b$.
% \end{claim}
% \begin{proof}
% Clearly $q \le d$, since every cell in $\Phi$ (which is a subset of $\Gamma$) must have a vertex in $S$. Also, $4b/k < 0.01b$ because $k$ is sufficiently large. Suppose for a contradiction that $d \ge 4b/k$ (so since $b \ge k$, we have $d-1 \ge 3b/k$). Then the quantity in \eqref{eq:b3} exceeds the previous quantity by \[\frac{d(d-1)}{2}(k-4) \ge \frac{(k-4)(d-1)}{2}q > bq,\] as long as $k, b$ are sufficiently large. This is a contradiction.
% \end{proof}

\begin{claim}
$q \le d < 0.031b$.
\end{claim}
\begin{proof}
Clearly $q \le d$, since every cell in $\Phi$ (which is a subset of $\Gamma$) must have a vertex in $S$. Suppose for a contradiction that $d \ge 0.031b$. Then we also have $d-1 \ge 0.021b$. Then the quantity in \eqref{eq:b3} exceeds the previous quantity by \[\frac{d(d-1)}{2}(k-4) \ge \frac{(k-4)(d-1)}{2}q \ge \frac{96(0.021b)}{2}q > bq,\] as long as $k, b$ are sufficiently large. This is a contradiction.
\end{proof}

\begin{claim}
$\ell < 0.025b$.
\end{claim}
\begin{proof}
Suppose $\ell \ge 0.025b$. Then $\ell-1 \ge 0.015b$. Then, \eqref{eq:b7} exceeds the previous line by at least \[\frac{\ell(\ell-1)}{2}(k-3) > 0.000375(k-3)b^2 > 0.031b^2 \ge bq,\]
a contradiction.
\end{proof}

\begin{claim} \label{claim:sibS}
Fewer than $q/2$ cells contain two siblings both in $S$.
\end{claim}
\begin{proof}
Suppose otherwise. Then we have $d - d' \ge q/2$, so we can bound $d'(k-1) + q < d(k-1)$, and thus we can replace the $d(k-1) + q$ term in \eqref{eq:b3} with just $d(k-1)$ (and make the inequality strict). If we follow through the chain of inequalities, this means we can remove the $bq$ term in \eqref{eq:b7}, thus obtaining a contradiction (since the inequality was strict).
\end{proof}

Let $\Phi'$ then be the set of cells in $\Phi$ which contain only one vertex of $S$, so that $|\Phi'| > q/2$ by the preceding claim.

\begin{cor}
$\ell > q/2$. (In particular, $\ell$ is positive.)
\end{cor}
\begin{proof}
Recall that all cells in $\Phi$ contain siblings, so every cell in $\Phi'$ must contain a vertex in $T$. Thus $\ell = |T| > q/2$.
\end{proof}

\begin{claim} \label{claim:fewsibs}
Fewer than $0.1b$ special vertices have siblings.
\end{claim}
\begin{proof}
Suppose that at least $0.1b$ special vertices have siblings. Note that $b-1-a$ is equal to the sum over all cells of $t-1$, where $t$ is the number of special vertices in the cell. Since $t-1 \ge t/2$ whenever the cell has siblings, this means that $b - 1 - a \ge 0.05b$. Rearranging, $a \ge b - 0.05b - 1$, so the bound \eqref{eq:bddeg14} can be reduced by $(0.05b - \ell)(k-2) > 0.025(k-2)b$. Consequently, the bound \eqref{eq:bdZ} (and thus also \eqref{eq:b2}) can be reduced by $0.025 (k-2) \ell b > 0.0125(k-2)bq > bq$, a contradiction.
\end{proof}

\begin{cor} \label{cor:abig}
$a > 0.89b$.
\end{cor}
\begin{proof}
Each lonely vertex is in its own cell, so $a$, the number of cells, is at least the number of lonely siblings, which by \cref{claim:fewsibs} is greater than $0.9b-1 > 0.89b$.
\end{proof}

\begin{claim} \label{claim:vtxThighdeg}
One of the vertices in $T$ has degree greater than $ak - 3b$.
\end{claim}
\begin{proof}
Suppose otherwise, and that all vertices in $T$ have degree at most $ak-3b$. We have $ak - 3b \le a(k-1) + r - 2b$, so the bound of \eqref{eq:bddeg11} exceeds the number of edges deleted from each vertex in $T$ by at least $2b$. Thus the bound $\eqref{eq:bdZ}$ can be reduced by at least $2 \ell b > bq$, a contradiction.
\end{proof}

\begin{cor}
$\deg(1) > ak - 3b > 0.89kb - 3b$.
\end{cor}
\begin{proof}
Since $1$ has maximum degree in $[b]$, this follows directly from \cref{claim:vtxThighdeg} and \cref{cor:abig}.
\end{proof}

Define a vertex in $[b] \setminus (S \cup T)$ to be \textit{good} if its degree is less than $0.95kb - 4b$. In particular, if a vertex is good, then its degree is less than $\deg(1) - b$.

\begin{claim} \label{claim:manygood}
There are at least $0.334b$ good vertices.
\end{claim}
\begin{proof}
Suppose otherwise. Then, since $|S \cup T| = d + \ell \le 0.056b$, this means that at least $0.61b$ special vertices have degree at least $0.89kb - 4b$. This means that the total number of edges in $G$ is (subtracting $b(b-1)/2$ to account for possible double counting of edges)
\[0.61b(0.89kb - 4b) - \frac{b(b-1)}{2} > 0.54kb^2 - 3b^2 > 0.5kb^2 + b^2.\]
But we have already shown in \eqref{eq:b7} that the number of edges in $G$ is at most $\frac{b(b-1)}{2}(k-1) + bq \le 0.5kb^2 + 0.031b^2$, which is less than the above quantity, which is a contradiction.
\end{proof}

Now, consider any good vertex $v$. Recall that when we derived the bound \eqref{eq:bdY3}, we used that at most $d'(k-1)+q$ edges are deleted from vertex $1$ in steps 1 and 2. Therefore, since $v$ has degree at least $b$ less than that of vertex $1$, step 3 cannot increase the number of deleted edges from $v$ farther than $d'(k-1)+q-b$.

\begin{claim}\label{claim:goodhighdeg}
There are at most $0.031b$ good vertices $v$ such that at most $d'(k-1)+q-b$ edges are deleted from $v$ in step 2.
\end{claim}
\begin{proof}
Suppose otherwise. For any such vertex $v$, the number of deleted edges touching $v$ in steps 2 and 3 combined cannot exceed $d'(k-1)+q-b$. Thus, for each such $v$, we can lower the bound \eqref{eq:bdY3} by at least $b$. Since there are more than $0.01b$ such $v$, in total we can lower the bound \eqref{eq:bdY3} by at least $0.031b^2 > bq$, a contradiction.
\end{proof}

Now say a vertex is \textit{very good} if it is lonely and good, and more than $d'(k-1)+q-b$ edges are deleted from $v$ in step 2. By the argument preceding the previous claim, no edges are deleted from any very good vertex in step 3.

\begin{claim}
There are at least $0.203b$ very good vertices.
\end{claim}
\begin{proof}
This follows directly from \cref{claim:manygood}, \cref{claim:fewsibs}, and \cref{claim:goodhighdeg}.
\end{proof}

Now, for any very good vertex $v$, let $j_v$ be such that $d'k - d - j_v$ vertices are deleted from $v$ in step 2. Since $d'k - d - j_v \le d'(k-1)+q - j_v$, this means that we can reduce the bound in \eqref{eq:bdY3} by $j_v$ for each such $v$. Thus, the sum of $j_v$ over all very good vertices $v$ is at most $bq$. But note that $j_v$ counts the number of missing edges from $v$ to all vertices in $\Gamma$ except those in $S$. Thus, there are at most $bq$ missing edges from all good vertices $v$ to all vertices in $\Gamma$ except those in $S$. Therefore, we can deduce the following.

\begin{claim} \label{claim:fewunhappypairs}
There are at most $0.011bq$ pairs $(\gamma, v)$, where $\gamma \in \Phi'$ and $v$ is a very good vertex, such that $v$ has no edge to any vertex in $\gamma$ except the one that is in $S$.
\end{claim}
\begin{proof}
Any such pair $(\gamma, v)$ contains $k-1$ missing edges from $v$ to the vertices in $\gamma$ except the ones in $S$, so by the preceding argument, there are at most $bq/(k-1) < 0.011bq$ such pairs.
\end{proof}

Define a pair $(\gamma, \beta)$ of cells, where $\gamma \in \Phi'$ and $\beta$ contains a lonely vertex $v$, to be \textit{happy} if there is an edge from $v$ to some vertex in $\gamma$ (except the vertex in $\gamma$ that is in $S$). There are at least $(q/2)(0.203b)$ pairs $(\gamma, \beta)$ where $\gamma \in \Phi'$ and $\beta$ is very good, so by \cref{claim:fewunhappypairs}, there are at least $0.09bq$ happy pairs. 

We will now show that happy pairs can be used to reduce the bound in \eqref{eq:bddeg11} in order to obtain a contradiction. In order to do so, we must first introduce some notation to measure how much the bound can be reduced.

We had used \eqref{eq:bddeg11} as a bound for the number of edges deleted from any $w \in T$ in step 4. Note the edges deleted from $w$ in step 4 are exactly those which are not to any edge in $[b] \setminus T$ (since the edges to $[b] \setminus T$ were deleted in steps 1 and 2). To this end, define \[F(w) = a(k-1) - e(w, (\alpha_1 \cup \dots \cup \alpha_a) \setminus ([b] \setminus T)),\] where the latter term is just the number of edges from $w$ to all cells, excluding those edges to $[b] \setminus T$. We also define \[f(w) = e(w, R),\] recalling that the exterior $R$ is defined as the complement of $\alpha_1 \cup \dots \cup \alpha_a \cup \{1\}$ (and thus is also disjoint from $[b] \setminus T$). Then, the total number of edges deleted from $w$ in step 4 is
\begin{equation*}
a(k-1) - F(w) + f(w).
\end{equation*}
Consequently, we can lower the bound on the number of edges deleted from $w$ given by \eqref{eq:bddeg11} (and thus also lower the bound \eqref{eq:bdZ}) by $F(w) - f(w)$, for each $w \in T$. Let \[E(w) = \max \{F(w) - f(w), 0\}.\] Since the bound \eqref{eq:bdZ} cannot be lowered by more than $bq$, we have
\begin{equation} \label{eq:bdE1}
\sum_{w \in T} E(w) \le bq.
\end{equation}
Now, define the function $F(w, \beta)$ to be the contribution of the cell $\beta$ toward $F(w)$. Formally, we define
\[F(w, \beta) = (k-1) - e(w, \beta \setminus ([b] \setminus T)).\]
Since each cell has at least one vertex in $[b] \setminus T$, we have $|\beta \setminus ([b] \setminus T)| \le k-1$, so $F(w, \beta)$ is always nonnegative. Then, we have by definition that
\[F(w) = \sum_{i \in [a]} F(w, \alpha_i).\]
Now, similarly define
\[E(w, \beta) = \max\{F(w, \alpha_i) - f(w), 0\}.\]
Then, we have 
\[E(w) \ge \sum_{i \in [a]} E(w, \alpha_i).\]
This is because if any term in the sum is positive, then the right hand side is immediately less than $F(w)$ by at least $f(w)$. Therefore, \eqref{eq:bdE1} implies that
\begin{equation*}
\sum_{w \in T,\, i \in [b]} E(w, \alpha_i) \le bq.
\end{equation*}
Furthermore, for a block $\gamma \in \Gamma$, define $E(\gamma, \beta)$ to be the sum over all $w \in \gamma$ which are also in $T$ of the quantity $E(w, \beta)$. Then, we have
\begin{equation} \label{eq:bdE3}
\sum_{\gamma \in \Gamma,\, i \in [b]} E(\gamma, \alpha_i) \le bq.
\end{equation}

Now we will bound the value of $E(\gamma, \beta)$ when $(\gamma, \beta)$ is a happy pair. First, we will need the following lemma.

\begin{lemma} \label{lem:nomatching}
If $(\gamma, \beta)$ forms a happy pair, then there exists no matching in $G$ consisting only of edges from $\gamma$ to $\beta \cup R$ which covers all vertices in $\gamma \cap T$.
\end{lemma}
\begin{proof}
Suppose otherwise. We will construct a matching $M$ which covers $[b]$. Let $x$ denote the vertex in $\gamma$ which is also in $S$ (there is only one since $\gamma \in \Phi'$), and let $v$ denote the special vertex in $\beta$.

Since $\gamma \in \Phi$, there exists an edge connecting $1$ with $x$; add this edge to $M$. Also, add all edges of the matching described in the lemma statement to $M$. Then $M$ now covers all special vertices in $\gamma$.

If $v$ is not already covered by $M$, then note that since $(\gamma, \beta)$ is happy, there must be an edge from $v$ to some vertex in $\gamma$ other than $x$. Then, we can just add this vertex to $M$, and if there is already an edge containing the other vertex, we can just delete it.

Now $M$ still covers all special vertices in $\gamma$, in addition to $v$ (which is the only special vertex in $\beta$) and 1. Also, all edges in $M$ only contain vertices in $\{1\} \cup \beta \cup \gamma \cup R$. Thus we can add all cells $\alpha_i$ other than $\beta, \gamma$ to $M$, so that $M$ is a matching which covers all special vertices, a contradiction.
\end{proof}

This allows us to bound $E(\gamma, \beta)$.

\begin{lemma} \label{lem:happyhighE}
If $(\gamma, \beta)$ is a happy pair, then $E(\gamma, \beta) \ge k-1$. 
\end{lemma}
\begin{proof}
Suppose $(\gamma, \beta)$ is happy. Then, applying Hall's marriage theorem along with \cref{lem:nomatching}, there must exist a set $W \subset \gamma \cap T$ with at most $|W|-1$ neighbors in $\beta \cup R$. Note that $1 \le |W| \le k-1$ (the second inequality is because $\gamma$ has one vertex in $S$ which is thus not in $T$). Then, for any $w \in W$, we have
\begin{align*}
E(w, \beta) \ge F(w, \beta) - f(w) 
&= (k-1) - e(w, \beta \setminus ([b] \setminus T)) - e(w, R) \\
&\ge (k-1) - e(w, \beta) - e(w, R) \\
&= (k-1) - e(w, \beta \cup R) \\
&\ge (k-1) - (|W|-1) \\
&= k - |W|.
\end{align*}
Summing over $w$ in $W$, this means that
\[E(\gamma, \beta) \ge \sum_{w \in W} E(w, \beta) \ge |W|(k-|W|).\]
Since $1 \le |W| \le k-1$, the above quantity is at least $k-1$.
\end{proof}

Finally, note that \cref{lem:happyhighE}, along with the previous observation that there are at least $0.09bq$ happy pairs, means that
\begin{equation*} 
\sum_{\gamma \in \Gamma,\, i \in [b]} E(\gamma, \alpha_i) \ge (0.09bq)(k-1) > bq.
\end{equation*}
This is at odds with \eqref{eq:bdE3}, so we finally have a contradiction. This completes the proof of \cref{prop:graph}.

\section{Concluding remarks}
%mention which techniques can be used to help solve erdos matching conjecture
This problem is closely related to the Erd\H{o}s matching conjecture, which states the following.
\begin{conj}[Erd\H{o}s matching conjecture, \cite{erdos}] 
Let $n, k, r$ be positive integers such that $n \ge kr$. Suppose that $\mcF$ is a family of $k$-subsets of $[n]$. If $\mcF$ has no matching of size $r$, then 
\[|\mcF| \le \max\left\{\binom{kr-1}{k}, \binom{n}{k} - \binom{n-r+1}{k}\right\}.\]
Here $\binom{kr-1}{k}$ is the number of $k$-subsets of $[kr-1]$, and $\binom{n}{k} - \binom{n-r+1}{k}$ is the number of $k$-subsets of $[n]$ which contain an element of $[r-1]$.
\end{conj}
The cases where $n \ge 5rk/3 - 2r/3$ and where $n \le (r+1)(k+\varepsilon_k)$ (where $\varepsilon_k > 0$ is some function of $k$) have been proven by Frankl and Kupavskii \cite{frankl-kupavskii-emc-1, frankl-emc-2}.
In these proofs, the authors crucially used the fact that $\mcF$ can be assumed to be a shifted family. We wonder whether our methods may be generalized to make progress on this conjecture. The notion of a blocking set is no longer helpful as is, since we care about matchings of size smaller than $n/k$. However, maybe it is possible to define a more general notion of a blocking set in the case that $n > kr$, and use methods similar to ours to show that such a blocking configuration must be of a certain form. 

We also wonder whether our methods can be used to prove \cref{thm:main} in the case where $k$ is small, in order to completely resolve \cref{conj:main}. In particular, the only obstacle to extending our proof to all $k$ is \cref{prop:graph}, the statement about graphs. Though we were only able to prove it for large enough $k$, based on examination of small cases, we conjecture that this statement is true for small $k$ as well. 
\begin{conj}
The statement of \cref{prop:graph} holds whenever $k \ge 4$.
\end{conj}
This conjecture would immediately imply \cref{conj:main} for all $k \ge 4$, since the rest of our proof still holds up. We believe there may still be a simple proof of this conjecture, since the statement is primarily about ordinary graphs, which are substantially simpler to deal with than set systems.

%In the case $k=3$, the construction $\mc E'(3, n, b)$ corresponds to a second equality case in \cref{prop:graph}. Specifically, this equality case is the following: $a = b-1$, and $\alpha_1, \dots, \alpha_a$ again each contain exactly one element of $[2, b]$. For each $i$, pick $u_i \in \alpha_i$ such that $u_i \notin [2, b]$. Then, $G$ contains exactly the edges between all elements of $[b]$ and all $u_i$. 
In the case $k=3$, the construction $\mc E'(3, n, b)$ corresponds to a second equality case in \cref{prop:graph}, shown in \cref{fig:alteqcase}.
If this equality case is in fact the only other equality case of \cref{prop:graph} (as small cases seem to suggest), then one can use an identical argument to ours to prove that the optimal $\mcF$ is either $\mc E(3, n, b)$ or $\mc E'(3, n, b)$. By the computation in \cref{sec:sizeE}, we would then know that it must be $\mc E'(3, n, b)$, leading to the following conjecture.

\begin{conj}
For each positive integer $b$ there exists $s_0(b)$ such that the following holds: Let $n = ks$, where $s \ge s_0(b)$, and let $\mcF$ be a family of $3$-subsets of $[n]$ with no perfect matching and no blocking set of size less than $b$. Then,
\[|\mcF| \le |\mc E'(3, n, b)|.\]
\end{conj}

\begin{figure}[ht]
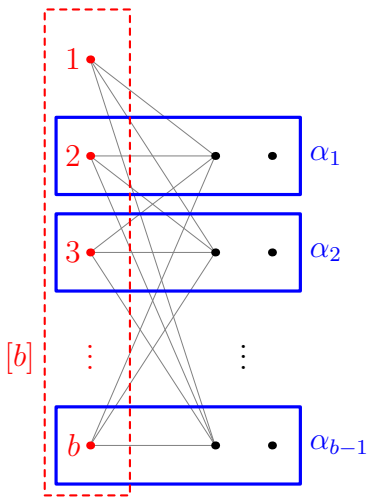

\centering
\begin{asy}
size(0, 6.5cm);
real r = 1.7; // y scale

pen edgepen = grey + 0.3;
for (int y1=0; y1<=4; ++y1) {
	for (int y2=0; y2<4; ++y2) {
		if (y2==1) continue;
		if (y1==1) continue;
		//for (int x=2; x<=4; ++x) {
			draw((-0.2, r*y1)--(2, r*y2), edgepen);
		//}
	}
}

for (int x=2; x<=3; ++x) {
	for (int y=0; y<=3; ++y) {
		if (y != 1) { dot( (x,r*y) ); }
	}
}
dot("$1$", (-0.2,4*r), dir(180), red);
dot("$2$", (-0.2,3*r), dir(180), red);
dot("$3$", (-0.2,2*r), dir(180), red);
label("$\vdots$", (-0.2,r), red);
dot("$b$", (-0.2,0), dir(180), red);

label("$\vdots$", (2.5,r));

pen alphapen = blue+1.2;
draw(box((-0.8,3.4*r), (3.5,2.6*r)), alphapen);
draw(box((-0.8,2.4*r), (3.5,1.6*r)), alphapen);
draw(box((-0.8,0.4*r), (3.5,-0.4*r)), alphapen);
label("$\alpha_1$", (3.5,3*r), dir(0), alphapen);
label("$\alpha_2$", (3.5,2*r), dir(0), alphapen);
label("$\alpha_{b-1}$", (3.5,0), dir(0), alphapen);

draw(box((-1,-0.4*r-0.2), (0.5, 4.4*r+0.2)), red+linetype("3 3")+0.8);
label("$[b]$", (-1, 0.9*r), dir(180), red);
\end{asy}
\caption{The other equality case of \cref{prop:graph} in the case where $k=3$}
\label{fig:alteqcase}
\end{figure}

In \cite{frankl-main}, Frankl also conjectured that the result of \cref{thm:frankl} is actually true for all $s \ge 6$ (as opposed to $s$ needing to be larger than some function of $k$). We may similarly expect that the dependence on $k$ may be removed in \cref{thm:main}. Thus we conjecture the following.
\begin{conj}
There exists a function $s_0(b)$ such that whenever $k, b, s$ are positive integers such that $k \ge 4$ and $s \ge s_0(b)$, the following holds: Let $n = ks$, and suppose that $\mcF$ is a family of $k$-subsets of $[n]$. Then, if $\mcF$ has no perfect matching and no blocking set of size less than $b$, then
\[|\mcF| \le |\mc E(k, n, b)|.\]
\end{conj}

\section{Acknowledgments}
This research was conducted at the University of Minnesota, Duluth REU and was supported by NSF-DMS grant 1949884 and NSA grant H98230-20-1-0009. Thanks to Joe Gallian for running the REU program. Thanks also to Zachary Chroman for helpful discussions, and to Evan Chen for reviewing a draft of this paper and providing suggestions and assistance in making diagrams.

\printbibliography[heading=bibintoc]

\end{document}